\documentclass[12pt, reqno]{amsart}
\usepackage{amsmath, amsthm, amscd, amsfonts, amssymb, graphicx}
\usepackage[bookmarksnumbered, colorlinks, plainpages]{hyperref}

\newtheorem{theorem}{Theorem}[section]

\newtheorem{corollary}[theorem]{Corollary}
\theoremstyle{definition}

\theoremstyle{remark}

\numberwithin{equation}{section}

\begin{document}
\title{A characterization of inner product spaces}

\author[M.S. Moslehian, J.M. Rassias]{Mohammad Sal Moslehian$^1$ and John M. Rassias$^2$}

\address{$^1$ Department of Pure Mathematics, Ferdowsi University of Mashhad, P. O. Box
1159, Mashhad 91775, Iran.} \email{moslehian@ferdowsi.um.ac.ir and
moslehian@ams.org}
\urladdr{\url{http://profsite.um.ac.ir/~moslehian/}}

\address{$^{2}$ Pedagogical Department, National and Capodistrian University of
Athens, Section of Mathematics and Informatics, 4, Agamemnonos str.,
Aghia Paraskevi, Attikis 15342, Athens, Greece.}
\email{jrassias@primedu.uoa.gr and Ioannis.Rassias@primedu.uoa.gr}
\urladdr{\url{http://www.primedu.uoa.gr/~jrassias/}}



\begin{abstract}
In this paper we present a new criterion on characterization of real
inner product spaces. We conclude that a real normed space $(X, \|\cdot\|)$ is an inner product space if
$$\sum_{\varepsilon_i \in \{-1,1\}} \left\|x_1 +
\sum_{i=2}^k\varepsilon_ix_i\right\|^2=\sum_{\varepsilon_i \in
\{-1,1\}} \left(\|x_1\| +
\sum_{i=2}^k\varepsilon_i\|x_i\|\right)^2\,,$$
for some positive integer $k\geq 2$ and all $x_1, \ldots, x_k \in X$. Conversely, if $(X, \|\cdot\|)$ is an inner product space, then the equality above holds for all $k\geq 2$ and all $x_1, \ldots, x_k \in X$
\end{abstract}

\maketitle


\section{Introduction}

There are a lot of significant natural geometric properties, which
fail in general normed spaces as non Euclidean spaces. Some of these
interesting properties hold just when the space is an inner product
space. This is the most important motivation for study of
characterizations of inner product spaces.

\noindent The first norm characterization of inner product spaces
was given by Fr\'echet \cite{FRE} in 1935. He proved that a normed
space $(X, \|\cdot\|)$ is an inner product space if and only if
$$\|x+y+z\|^2+\|x\|^2+\|y\|^2+\|z\|^2-\|x+y\|^2-\|y+z\|^2-\|x+z\|^2=0$$
for all $x, y, z \in X$. In 1936 Jordan and von Neumann \cite{J-V}
showed that a normed space $X$ is an inner product space if and only
if the parallelogram law $\| x-y\|^2+\| x+y\|\sp 2=2\| x\|^2+2\|
y\|^2$ holds for all $x, y \in X$. Later, Day \cite{DAY} showed that
a normed linear space $X$ is an inner product space if one requires
only that the parallelogram equality holds for $x$ and $y$ on the
unit sphere. In other words, he showed that the parallelogram
equality may be replaced by the condition $R=4$ ($\| x\|=1$, $\|
y\|=1$), where $R=\| x-y\|^2+\| x+y\|^2$. There are several
characterizations of inner product spaces introduced by many
mathematicians some of which are [1--16].

In this paper we present a new criterion on characterization of
inner product spaces and give an operator version of it. The notion of inner product space plays an essential role in quantum mechanics, since every physical system is associated with a Hilbert space and self-adjoint operators associated to a system represent physical quantities; see \cite{ZET}.

\section{Main result}
\begin{theorem}
Let $(X, \|\cdot\|)$ be a real normed space, $n$ be a positive real
number and $k\geq 2$ be a positive integer. If
$$R_{k,n}=\sum_{\varepsilon_i \in \{-1,1\}} \left\|x_1 +
\sum_{i=2}^k\varepsilon_ix_i\right\|^n$$ and
$$A_{k,n}=\sum_{\varepsilon_i \in \{-1,1\}} \left(\|x_1\| + \sum_{i=2}^k\varepsilon_i\|x_i\|\right)^n\,,$$
then a necessary and sufficient condition for that the norm $\|\cdot\|$ over $X$ is induced by an inner product is that\\
(I) $R_{k,n} \leq A_{k,n}$ if $n \geq 2$\\
and\\
(II) $R_{k,n} \geq A_{k,n}$ if $0 < n \leq 2$\\
for any $x_1, \ldots, x_k \in X$.
\end{theorem}
\begin{proof}

\textbf{Necessity.}\\
Assume that the norm $\|\cdot\|$ on $X$ is induced by an inner
product $\langle\cdot,\cdot\rangle$. Hence $\|x\|^2=\langle x,
x\rangle\,\,(x \in X)$. We have
\begin{eqnarray*}
R_{k,n}&=&\sum_{\varepsilon_i \in \{-1,1\}} \left\|x_1 + \sum_{i=2}^k\varepsilon_ix_i\right\|^n\\
&=& \frac{1}{2}\sum_{\varepsilon_i \in \{-1,1\}} \left(\left\|\sum_{i=1}^k\varepsilon_ix_i\right\|^2\right)^{n/2}\\
&=& \frac{1}{2}\sum_{\varepsilon_i, \varepsilon_j \in \{-1,1\}} \left(\sum_{i=1}^k \|x_i\|^2 + 2\sum_{1 \leq i < j \leq k}\varepsilon_i \varepsilon_j \langle x_i, x_j\rangle\right)^{n/2}\\
&=& \frac{1}{2}\sum_{\varepsilon_i, \varepsilon_j \in \{-1,1\}} \left(a + \sum_{1 \leq i < j \leq k}\varepsilon_{i}
\varepsilon_{j} a_{i,j} \cos(p_{i,j})\right)^{n/2}\\
&=& R_{k,n}(P)\,,\\
\end{eqnarray*}
where $a:=\sum_{i=1}^k \|x_i\|^2,  a_{i,j}:=2\|x_i\|\,\|x_j\|$ and
$p_{i,j}$'s are defined in such a way that $\langle x_i, x_j\rangle
= \|x_i\|\,\|x_j\|\,\cos(p_{i,j})$. Let $P$ denote the
$\frac{k^2-k}{2}$-tuple consisting of $p_{i,j}\,\,(1 \leq i < j \leq
k)$ by going row-by-row throughout the matrix
\begin{eqnarray*}
\left[\begin{array}{ccccc}
\star& \langle x_1,x_2\rangle& \langle x_1,x_3\rangle& \cdots &  \langle x_1,x_k\rangle\\
\star& \star& \langle x_2,x_3\rangle& \cdots &  \langle x_2,x_k\rangle\\
\vdots&\vdots&\vdots&\vdots& \vdots\\
\star& \star& \cdots & \star& \langle x_{k-1},x_k\rangle\\
\star& \star& \cdots & \star& \star
\end{array}\right]\,.
\end{eqnarray*}
For each fixed $1 \leq t < s \leq k$, we have
\begin{eqnarray*}
\frac{\partial R_{k,n}(P)}{\partial p_{t,s}}&=&
\frac{1}{2}\sum_{\varepsilon_i, \varepsilon_j \in \{-1,1\}}\left[ -
\frac{n}{2}\left(a
+ \sum_{1 \leq i < j \leq k}\varepsilon_{i} \varepsilon_{j} a_{i,j} \cos(p_{i,j})\right)^{\frac{n-2}{2}} \varepsilon_t\varepsilon_s a_{t,s} \sin(p_{t,s})\right]\\
&=& \frac{n}{4} \varphi(P)a_{t,s}\sin(p_{t,s})\,,
\end{eqnarray*}
in which $$\varphi(P):= \sum_{\varepsilon_i, \varepsilon_j \in
\{-1,1\}}- \varepsilon_t \varepsilon_s\left(a + \sum_{1 \leq i < j
\leq k}\varepsilon_{i} \varepsilon_{j} a_{i,j}
\cos(p_{i,j})\right)^{\frac{n-2}{2}}\,.$$ The solution of the system
of equations $\frac{\partial R_{k,n}(P)}{\partial p_{t,s}}$ where
$t,s$ run throughout $1 \leq t < s \leq k$ is $P_0=(K_1\pi, \cdots,
K_{\frac{k^2-k}{2}}\pi)$, where $K_1, \ldots, K_{\frac{k^2-k}{2}}
\in \{0, \pm1, \pm2, \ldots\}$. We use the second partial test to
show that $P_0$ is an extremum point of $R_{k,n}(P)$. For $(u,v)
\neq (t,s)$, we have
\begin{eqnarray}\label{1}
\frac{\partial^2 R_{k,n}(P)}{\partial p_{u,v}\partial p_{t,s}}&=& \frac{\partial}{\partial p_{u,v}}\left(\frac{n}{4} \varphi(P)a_{t,s}\sin(p_{t,s}) \right)\nonumber\\
&=&\frac{n}{4}a_{t,s}\sin(p_{t,s})\frac{\partial}{\partial p_{u,v}}\varphi(P)
\end{eqnarray}
and
\begin{eqnarray}\label{2}
\frac{\partial^2 R_{k,n}(P)}{\partial p_{t,s}^2}&=& \frac{\partial}{\partial p_{t,s}}\left(\frac{n}{4} \varphi(P)a_{t,s}\sin(p_{t,s}) \right)\nonumber\\
&=&\frac{n}{4}a_{t,s}\sin(p_{t,s})\frac{\partial}{\partial p_{t,s}}\varphi(P) + \frac{n}{4} \varphi(P)a_{t,s}\cos(p_{t,s})\,.
\end{eqnarray}
It follows from \eqref{1} that
$$\frac{\partial^2 R_{k,n}}{\partial p_{u,v}\partial p_{t,s}}(P_0)=0$$
and from \eqref{2} that
\begin{eqnarray*}
\frac{\partial^2 R_{k,n}}{\partial p_{t,s}^2}(P_0)&=& \frac{n}{4}a_{t,s}\varphi(P_0)\\
&=&\frac{n}{4}a_{t,s} \sum_{\varepsilon_i, \varepsilon_j \in \{-1,1\}}- \varepsilon_t \varepsilon_s\left(a + \sum_{1 \leq i < j \leq k}\varepsilon_{i} \varepsilon_{j} a_{i,j}\right)^{\frac{n-2}{2}}\\
&:=& \gamma_{t,s}\,.
\end{eqnarray*}
We also consider the determinants
\begin{eqnarray*}
D_1(P_0)&:=&\frac{\partial^2 R_{k,n}}{\partial p_{1,2}^2}(P_0)=\gamma_{1,2}\\
D_2(P_0)&:=&\left|\begin{array}{cc}\frac{\partial^2
R_{k,n}}{\partial p_{1,2}^2}(P_0)&\frac{\partial^2 R_{k,n}}{\partial
p_{1,2}\partial p_{1,3}}(P_0)\\\frac{\partial^2 R_{k,n}}{\partial
p_{1,3}\partial p_{1,2}}(P_0)&
\frac{\partial^2 R_{k,n}}{\partial p_{1,3}^2}(P_0)\end{array}\right|=\left|\begin{array}{cc}\gamma_{1,2}&0\\0&\gamma_{1,3}\end{array}\right|=\gamma_{1,2}\gamma_{1,3}\\
\vdots\\
D_{\frac{k^2-k}{2}}(P_0)&:=&\left|\begin{array}{cccc}\frac{\partial^2 R_{k,n}}{\partial p_{1,2}^2}(P_0)&\frac{\partial^2 R_{k,n}}{\partial p_{1,2}\partial p_{1,3}}(P_0)& \cdots& \frac{\partial^2 R_{k,n}}{\partial p_{1,2}\partial p_{k-1,k}}(P_0)\\
\frac{\partial^2 R_{k,n}}{\partial p_{1,3}\partial p_{1,2}}(P_0)&\frac{\partial^2 R_{k,n}}{\partial p_{1,3}^2}(P_0)& \cdots &\frac{\partial^2 R_{k,n}}{\partial p_{1,3}\partial p_{k-1,k}}(P_0)\\
\vdots&\vdots&\vdots&\vdots\\
\frac{\partial^2 R_{k,n}}{\partial p_{k-1,k}\partial p_{1,2}}(P_0)&\frac{\partial^2 R_{k,n}}{\partial p_{k-1,k}\partial p_{1,3}}(P_0)& \cdots& \frac{\partial^2 R_{k,n}}{\partial p_{k-1,k}^2}(P_0)\end{array}\right|\\
&=&\left|\begin{array}{cccc}\gamma_{1,2}&0&0&0\\0&\gamma_{1,3}&0&0\\ \vdots&\vdots&\vdots&\vdots\\0&0&0&\gamma_{k-1,k}
\end{array}\right|=\gamma_{1,2}\gamma_{1,3} \cdots \gamma_{k-1,k}\,.
\end{eqnarray*}
It is not hard to see that for each $t,s$, $\gamma_{t,s} <0$ if
$n>2$ and $\gamma_{t,s}>0$ if $0<n <2$. Hence $(-1)^iD_i(P_0)
>0\,\, (i=1, 2, \cdots, (k^2-k)/2)$ for $n>2$, whence, by utilizing
the second partial test, we infer that
\begin{eqnarray*}
\max_P R_{k,n}(P)&=&\max_{P_0} R_{k,n}(P_0)\\
&=&\frac{1}{2}\sum_{\varepsilon_i, \varepsilon_j \in \{-1,1\}} \left(2\sum_{1 \leq i, j \leq k}\varepsilon_{i} \varepsilon_{j} \|x_i\|\,\|x_j\|\right)^{n/2}\\
&=&\frac{1}{2}\sum_{\varepsilon_i \in \{-1,1\}} \left[\left( \sum_{i=1}^k\varepsilon_{i} \|x_i\|\right)^2\right]^{n/2}\\
&=& \frac{1}{2}\sum_{\varepsilon_i \in \{-1,1\}} \left( \sum_{i=1}^k\varepsilon_{i} \|x_i\|\right)^n\\
&=&\sum_{\varepsilon_i \in \{-1,1\}} \left(\|x_1\| + \sum_{i=2}^k\varepsilon_{i} \|x_i\|\right)^n\\
&=&A_{k,n}\,,
\end{eqnarray*}
which yields (I). Similarly, $D_i(P_0) >0\,\, (i=1, 2, \cdots,
(k^2-k)/2)$ for $0<n<2$, whence, by utilizing the second partial
test, we deduce that
$$\min_P R_{k,n}(P)=\min_{P_0} R_{k,n}(P_0)=A_{k,n}\,,$$
which gives us (II).\\

\textbf{Sufficiency.} \\
Assume that condition (I) to be held. The continuity of the function $n \mapsto \| \cdot\|^n$ implies that
$$R_{k,2} \leq A_{k,2}=k2^{k-1}$$
for $\|x_1\|= \cdots=\|x_k\|=1$. From the pertinent sufficient condition of M.M. Day, it can be proved the following criterion \cite{DAY}:\\
``The necessary and sufficient condition for a norm defined over a
vector space $X$ to spring from an inner product is that $R_{k,2}
\leq k2^{k-1}$ where $k\geq 2$ is a positive integer and  $\|x_1\|=
\cdots=\|x_k\|=1$''. Due to this condition holds, we conclude that
the norm $\|\cdot\|$ on $X$ can be deduced from an inner product.

Similarly, if condition (II) holds, then we get  $$R_{k,2} \geq
A_{k,2}=k2^{k-1}$$ for $\|x_1\|= \cdots=\|x_k\|=1$. Applying the
same statement as the above criterion except that $R_{k,2} \geq
k2^{k-1}$, we conclude that the norm $\|\cdot\|$ on $X$ can be
deduced from an inner product.
\end{proof}

\begin{corollary}\label{cor}
A normed space $(X, \|\cdot\|)$ is an inner product space if
\begin{eqnarray}\label{referee}
\sum_{\varepsilon_i \in \{-1,1\}} \left\|x_1 +
\sum_{i=2}^k\varepsilon_ix_i\right\|^2=\sum_{\varepsilon_i \in
\{-1,1\}} \left(\|x_1\| +
\sum_{i=2}^k\varepsilon_i\|x_i\|\right)^2
\end{eqnarray}
for some $k\geq 2$ and all $x_1, \ldots, x_k \in X$. The converse is true if \eqref{referee} holds for all $k\geq 2$ and all $x_1, \ldots, x_k \in X$.
\end{corollary}

We can have an operator version of Corollary above. In fact a
straightforward computation shows that
\begin{corollary}\label{cor} Let $k \geq 2$ and $T_1, T_2, \ldots, T_k$
be bounded linear operators acting on a Hilbert space. Then
$$\sum_{\varepsilon_i \in \{-1,1\}} \left|T_1 +
\sum_{i=2}^k\varepsilon_iT_i\right|^2=2^{k-1}\sum_{i=1}^k|T_i|^2=\sum_{\varepsilon_i
\in \{-1,1\}} \left(|T_1| +
\sum_{i=2}^k\varepsilon_i|T_i|\right)^2\,,$$
where $|T|=(T^*T)^{1/2}$
denotes the absolute value of $T$.
\end{corollary}



\begin{thebibliography}{99}

\bibitem{A-U} J. Alonso and A. Ull\'{a}n, \textit{Moduli in normed linear spaces and characterization of inner product spaces}, Arch. Math. (Basel) \textbf{59} (1992), no. 5, 487--495.

\bibitem{AST} C. Alsina, J. Sikorska and M.S. Justyna, \textit{Norm derivatives and characterizations of inner product spaces},
Hackensack, NJ: World Scientific, 2010.

\bibitem{AMI} D. Amir, \textit{Characterizations of inner product spaces}, Operator Theory: Advances and Applications, 20 Birkh\"auser Verlag, Basel, 1986.

\bibitem{B-C} M. Baronti and E. Casini, \textit{Characterizations of inner product spaces by orthogonal vectors}, J. Funct. Spaces Appl. \textbf{4} (2006), no. 1, 1--6.

\bibitem{D-M} F. Dadipour  and M.S. Moslehian, \textit{An approach to operator Dunkl--Williams inequality}, J. Math. Anal. Appl. (to appear).

\bibitem{DAY} M.M. Day, \textit{Some characterizations of inner-product spaces}, Trans. Amer. Math. Soc. \textbf{62} (1947), 320--337.

\bibitem{D-A-F} C.R. Diminnie, E.Z. Andalafte and R. Freese, \textit{Triangle congruence characterizations of inner product spaces}, Math. Nachr. \textbf{144} (1989), 81--86.

\bibitem{DRA} S.S. Dragomir, \textit{Some characterizations of inner product spaces and applications}, Studia Univ. Babes-Bolyai Math. \textbf{34} (1989), no. 1, 50--55.

\bibitem{FRE} M. Fr\'echet, \textit{Sur la définition axiomatique d'une classe d'espaces vectoriels distanci\'es applicables vectoriellement sur
l'espace de Hilbert}, Ann. of Math. \textbf{(2) 36} (1935), no. 3, 705--718.

\bibitem{J-V} P. Jordan and J. von Neumann, \textit{On inner products in linear, metric spaces}, Ann. of Math. \textbf{(2) 36} (1935), no. 3, 719--723.

\bibitem{M-R} M.S. Moslehian and J.M. Rassias, \textit{A characterization of inner product spaces concerning an Euler--Lagrange identity}, Commun. Math. Anal. \textbf{8} (2010),  no. 2, 16--21.

\bibitem{N-P} K. Nikodem and Z. Pales, \textit{Characterizations of inner product spaces by strongly convex functions}, Banach J. Math. Anal. (to appear).

\bibitem{PAP} P.L. Papini, \textit{Inner products and norm derivatives}, J. Math. Anal. Appl. \textbf{91} (1983), no. 2, 592--598.

\bibitem{JRAS1} J.M. Rassias, \textit{Two new criteria on characterizations of inner products} Discuss. Math. \textbf{9} (1988), 255--267 (1989).

\bibitem{JRAS2} J.M. Rassias, \textit{Four new criteria on characterizations of inner products}, Discuss. Math. \textbf{10} (1990), 139--146 (1991).

\bibitem{TRAS} Th.M. Rassias, \textit{New characterizations of inner product spaces}, Bull. Sci. Math. (2) \textbf{108} (1984), no. 1, 95--99.

\bibitem{ZET} N. Zettili, \textit{Quantum Mechanics: Concepts and Applications}, Chichester: John Wiley \& Sons, Ltd, 2009.

\end{thebibliography}
\end{document}